\documentclass{amsart}
\usepackage{amsmath}
\usepackage{amssymb}
\usepackage{mathrsfs}
\usepackage{pgf,tikz}
\usepackage{tkz-graph}
\usepackage{graphicx}
\usepackage{systeme,mathtools}
\usepackage{amsthm,enumitem}
\usepackage{systeme}
\usetikzlibrary{shapes,mindmap,trees, positioning}

\newtheorem{theorem}{Theorem}[section]

\newtheorem{lemma}[theorem]{Lemma}
\newtheorem{corollary}[theorem]{Corollary}
\newtheorem{proposition}[theorem]{Proposition}
\theoremstyle{definition}
\newtheorem{definition}[theorem]{Definition}
\newtheorem{remark}{Remark}
\newtheorem{example}{Example}

\begin{document}
	
	\title[Leavitt path algebras in which every Lie ideal is an ideal]{Leavitt path algebras in which every Lie ideal is an ideal and applications}
	\author[Hu\`{y}nh Vi\d{\^{e}}t Kh\'{a}nh]{Hu\`{y}nh Vi\d{\^{e}}t Kh\'{a}nh}
	\address{Department of Mathematics and Informatics, HCMC University of Education, 280 An Duong Vuong Str., Dist. 5, Ho Chi Minh City, Vietnam}
	\email{khanhhv@hcmue.edu.vn}
	\maketitle
	\begin{abstract} 
		In this paper, we classify all Leavitt path algebras which have the property that every Lie ideal is an ideal. As an application, we show that Leavitt path algebras with this property provide a class of locally finite, infinite-dimensional Lie algebras whose locally solvable radical is completely determined. This particularly gives us a new class of semisimple Lie algebras over a field of prime characteristic.
	\end{abstract}
	
	\section{Introduction}
		The Leavitt path algebra $L_K(E)$ of a graph $E$ with coefficients taken from a field $K$ was first introduced in \cite{Pa_abrams-pino-05} and attracted the attention of many researchers. Those rings were simultaneously and independently defined in \cite{Pa_ara-moreno-pardo-07},  though it appeared in print two years after \cite{Pa_abrams-pino-05}. Under certain assumptions on the graph $E$ and the field $K$, the Leavitt path algebra $L_K(E)$ provides many examples of well-known algebraic structures. For instance, the classical Leavitt $K$-algebra $L_K(1,n)$ for $n\geq 2$, the full $n \times n$ matrix algebra $\mathbb{M}_n(K)$, and the Toeplitz $K$-algebra $\mathscr{T}_K$ are, respectively, the Leavitt path algebras of the “rose with $n$ petals” graph $R_n$ ($n \geq 2$), the oriented line graph $A_n$ with $n$ vertices, and the Toeplitz graph $E_T$ (see \cite[Section 1.3]{Bo_abrams-ara-molina-LPA}).
	
		Within the past few years, Lie structures arising from $L_K(E)$ have been substantially investigated by many authors. For example, one can obtain the Lie algebra by considering $L_K(E)$ with the same $K$-vector space structure and the Lie bracket given by $[a,b]=ab-ba$, where $ a,b \in L_K(E)$. (Please see Subsection \ref{subsection_2.2} for some basic concepts concerning Lie algebras). In this direction, the paper \cite{Pa_nam-zhang-22} completely identified the graph $E$ and the field $K$ for which $L_K(E)$ is Lie solvable. At the other extreme, necessary and sufficient conditions for the Lie algebra $[L_K(E),L_K(E)]$ to be simple were given in \cite{Pa_abrams-mesyan-12} and \cite{Pa_alahmedi-alsulami-16}, under the assumption that $E$ is row-finite. In a similar line, in \cite{Pa_alahmadi-alsulami-16-2} the simplicity of the Lie algebra of skew symmetric elements of $L_K(E)$ was also studied.
	
		Motivated by \cite{Pa_abrams-mesyan-12}, Z. Mesyan gave in \cite{Pa_mesyan-13} a full description of the elements of $[L_K(E),L_K(E)]$. With such a description in hand, he also determined the field $K$ and the graph $E$ for which $L_K(E)$ is a commutator ring. (Recall that an associative ring $R$ is called a commutator ring if $R$ equals the additive subgroup $[R,R]$ generated by its additive commutators.) It is interesting to know that all commutator Leavitt path algebras have the further property that every Lie ideal is also an ideal. (The term Lie ideal of $L_K(E)$ means a $K$-subspace $U \subseteq L_K(E)$ such that $[U,L_K(E)] \subseteq U$.) At this point, a question is naturally raised: Can we identify the graph $E$ and the field $K$ for which $L_K(E)$ has the property that every Lie ideal is an ideal? This is the main purpose of the current paper. It seems to be very hard to give a complete description of Lie ideals of $L_K(E)$, while the (ring-theoretic) ideals of $L_K(E)$ are completely determined. Thus, giving the answer to the above question means that we provide a class of Leavitt path algebras in which all Lie ideals can be completely identified. As an application, among other interesting results, we also obtain from these Leavitt path algebras a new class of semisimple Lie algebras.
	
		Apart from the Introduction section, the paper is divided into four parts. Section 2 provides some basic definitions and certain results concerning Leavitt path algebras and Lie algebras that we use in the sequel. 
	
		Section 3 serves as the main purpose of the current paper. In Theorem \ref{theorem_3.8}, we give necessary and sufficient conditions on $E$ and $K$ for $L_K(E)$ to have the property that every Lie ideal is an ideal. As it will turn out, to ensure that $L_K(E)$ has such a property, the field $K$ must have the characteristic $p > 0$ and the graph $E$ falls into three kinds of graphs. In other words, there are exactly three classes of Leavitt path algebras having the property that every Lie ideal is an ideal, one of which is the class of commutator Leavitt path algebras studied by Z. Mesyan in \cite{Pa_mesyan-13}. Accordingly, Theorem \ref{theorem_3.8} serves as two purposes:  it first generalizes \cite[Theorem 27]{Pa_mesyan-13} and then permits us to construct a class of Lie algebras in which all Lie ideals can be completely determined.
	
		Finally, in Section 4, we illustrate an application of the three classes of Leavitt path algebras $L_K(E)$ described in Theorem \ref{theorem_3.8}. In order to obtain this, we first show that $L_K(E)$ is locally finite as a $K$-algebra if and only if $E$ is acyclic. We also show that the two conditions ``being local Lie solvable" and ``being Lie solvable" are equivalent for locally finite $K$-algebra $L_K(E)$. In particular, we show that if $L_K(E)$ has the property that every Lie ideal is an ideal, then it is a locally finite $K$-algebra. In consequence, such a $L_K(E)$ with the property that every Lie ideal is an ideal becomes a Lie algebra (under bracket $[a,b]=ab-ba$) which is locally finite, and infinite-dimensional over the field $K$ of characteristic $p > 0$. Finally, we calculate the locally solvable radical of this Lie algebra. The calculation allows us to establish a new class of infinite-dimensional, locally finite, semisimple Lie algebras over a field of prime characteristic.
	
	\section{Preliminaries}
	\subsection{Directed graphs and Leavitt path algebras.}
		Throughout this paper, basic notations and conventions are taken from \cite{Bo_abrams-ara-molina-LPA}. In the following, we outline some concepts and notations that will be used in this paper. A directed graph $E=(E^0, E^1, r, s)$ consists of two sets $E^0$ and $E^1$ together with maps $r, s: E^1\to E^0$. The elements of $E^0$ are called \textit{vertices} and the elements of $E^1$ \textit{edges}. We sometimes use the notations $r_E$ and $s_E$ rather than $r$ and $s$ to emphasize the maps we are indicating are those of $E$. The graph $E$ is \textit{finite} if both $E^0$ and $E^1$ are finite. 
	
		A vertex $v$ is called a \textit{sink} if it emits no edges (i.e., $s^{-1}(v)=\varnothing$), while it is called an \textit{infinite emitter }if it emits infinitely many edges (i.e., $s^{-1}(v)$ is infinite). A vertex $v$ is said to be \textit{regular} if it is neither a sink nor an infinite emitter. The set of all regular vertices of $E$ is denoted by ${\rm Reg}(E)$. A \textit{finite path} $\mu$ of length $\ell(\mu):=n\ge 1$ is a finite sequence of edges $\mu=e_1e_2\dots e_n$ with $r(e_i) =s(e_{i+1})$ for all $1\leq i\leq n-1$. We set $s(\mu):=s(e_1)$ and $r(\mu):=r(e_n)$. The set of all finite paths in $E$ is denoted by ${\rm Path}(E)$. Let $\mu=e_1\dots e_n\in {\rm Path}(E)$. If $v=s(\mu)=r(\mu)$, then we say that $\mu$ is a \textit{closed path based at} $v$. If moreover, $s(e_j)\ne v$ for every $j>1$ then $\mu$ is a \textit{simple closed path based at} $v$. If $\mu$ is  a closed path based at $v$ and $s(e_i)\ne s(e_j)$ for every $i\ne j$, then $\mu$ is called a \textit{cycle based at} $v$. Assume that $\mu=e_1\dots e_n$ is a cycle based at the vertex $v$. Then for each $1\le i\le n$, the path $e_ie_{i+1}\dots e_ne_1\dots e_{i-1}$ is a cycle based at $s(e_i)$. We call the collection of cycles $\{\mu_i\}$ based at $s(e_i)$ the \textit{cycle of} $\mu$. A \textit{cycle $c$ based at $v$} is a set of paths consisting of the cycle of $\mu$ for  $\mu$ some cycle based at a vertex $v$. The \textit{length of a cycle} $c$ is the length of any path in $c$. In particular, a cycle of length $1$ is called a \textit{loop}. A graph $E$ is said to be \textit{acyclic} if it does not have any closed path based at any vertex of $E$, or equivalently if it does not have any cycle based at any vertex of $E$. 
	
		Let $H$ be a subset of $E^0$. We say that $H$ is \textit{hereditary}  if whenever $u\in H$ and there is a path from $u$ to some vertex $v$, then $v\in H$; and $H$ is \textit{saturated} if, for any regular vertex $v$, $r(s^{-1}(v))\subseteq H$ implies $v\in H$. Assume that $H$ is a hereditary and saturated subset of $E^0$. A vertex $w$ is called a \textit{breaking vertex} of $H$ if $w \in E^0 \backslash H$ is an infinite emitter such that $1\leq |s^{-1}(w)\cap r^{-1}(E^0\backslash H)|<\infty$. The set of all breaking vertices of $H$ is denoted by $B_H$. For each $w \in B_H$, we set
		$$
		w^H=w-\sum_{s(e)=w, \;r(e)\notin H}ee^*.
		$$
		
		An \textit{admissible pair} of a graph $E$ is an ordered pair $(H,S)$, where $H$ is hereditary and saturated subset of $E^0$ and $S \subseteq B_H$. 
		\begin{definition}[Leavitt path algebra]\label{definition_1.3}
			Let $E$ be a directed  graph and $K$ a field. We define a set $\left( E^1\right)^*$ consisting of symbols of the form $\{e^*| e\in E^1\}$.   The \textit{Leavitt path algebra of $E$ with coefficients in $K$}, denoted by $L_K(E)$, is the associative $K$-algebra generated by $E^0\cup E^1\cup (E^1)^*$, subject to the following relations:
			\begin{enumerate}[]
				\item[(V)] $vv'=\delta_{v,v'}v$ for all $v,v'\in E^0$,
				\item[(E1)] $s(e)e=er(e)=e$ for all $e\in E^1$,
				\item[(E2)] $r(e)e^*=e^*s(e)=e^*$ for all $e\in E^1$,
				\item[(CK1)] $e^*f=\delta_{e,f}r(e)$ for all $e,f\in E^1$, and
				\item[(CK2)] $v=\sum_{\{e\in E^1|s(e)=v\}}ee^*$ for every $v\in{\rm Reg}(E)$.
			\end{enumerate}
		\end{definition}
	We let $r(e^*)$ denote $s(e)$, and we let $s(e^*)$ denote $r(e)$. If  $\mu=e_1\dots e_n\in {\rm Path}(E)$, then we denote by $\mu^*$ the element $e_n^*\dots e_2^*e_1^*$ of $L_K(E)$. 
	
	For an admissible pair $(H,S)$, we denote by $I(H, S)$ the ideal of $L_K(E)$ generated by the set  $H\cup \{v^H: v\in S\}$. 
	
	Let $G$ be an abelian group which is written additively, and $A$ an algebra over a field $K$. We say that $A$ is \textit{$G$-graded} if there exists a family $\{A_{\sigma}\}_{\sigma\in G}$ of $K$-subspaces of $A$ such that 
	$$
	A=\bigoplus_{\sigma\in G}A_{\sigma},
	\text{ as } K\text{-subspaces, and } A_{\sigma}\cdot A_{\tau}\subseteq A_{\sigma+\tau}, \text{ for each } \sigma,\tau\in G.
	$$
	A subset $X$ of $A$ is said to be graded if for every $x=\sum_{\sigma\in G}x_{\sigma}\in X$, with $x_{\sigma}\in A_{\sigma}$, each $x_{\sigma}\in X$. Let $\mathbb{Z}$, $\mathbb{N}$, and $\mathbb{Z}_+$ denote the set of all integers, natural numbers, and positive integers respectively. It was established in \cite[Corollary 2.1.5]{Bo_abrams-ara-molina-LPA} that the Leavitt path algebra $L_K(E)$ has a natural $\mathbb{Z}$-grading given by the length of the monomials. That is, $L_K(E)=\oplus_{n\in\mathbb{Z}}L_n$, where 
	$$
	L_n={\rm span}_K\{\gamma\lambda^*|\lambda,\gamma\in{\rm Path}(E) \text{ and } \ell(\gamma)-\ell(\lambda)=n\}.
	$$
	Throughout this paper, unless otherwise specified, the grading structure of $L_K(E)$ we use is the natural $\mathbb{Z}$-grading just mentioned above.
	
	\subsection{Lie algebras}\label{subsection_2.2}
		In our paper, we will particularly study a certain Lie structure arising from Leavitt path algebras. Thus, it is reasonable to present here some basic concepts concerning Lie algebras. 
		\begin{definition}[Lie algebra]
			Let $K$ be a field. A \textit{Lie algebra} over $K$ is a $K$-vector space $L$, together with a bilinear map, the \textit{Lie bracket} $[\;,\;]: L\times L\to L$, $(x,y)\mapsto[x,y]$ such that the following axioms are satisfied:
			\begin{enumerate}
				\item $[x,x]=0$,
				\item $[x,[y,z]]+[y,[z,x]]+[z,[x,y]]=0$, \hspace{0.5cm} (Jacobi identity)
			\end{enumerate}
			for every $x,y,z$ in $L$. 
		\end{definition}
	
		A \textit{Lie subalgebra} of $L$ is defined to be a $K$-subspace $L'$ of $L$ such that $[a,b]\in L'$ for all $a,b\in L'$. On the other hand, a subspace $I$ of $L$ is called an ideal of $L$ if $[a,b]\in I$ for all $a\in L$ and $b\in I$. Given two ideals $I$ and $J$ of $L$, their product is the ideal
		$$
		[I,J]={\rm span}_K\{[a,b]\;|\;a\in I, b\in J\}.
		$$
		Let $(L,[\;,\;])$ be a Lie algebra over $K$. We define the \textit{derived series} of $L$ to be the series with terms
		$$
		L^{(0)}=L,\;L^{(1)}=[L,L],\;\dots,\; L^{(n)}=[L^{(n-1)},L^{(n-1)}].
		$$
		The Lie algebra $L$ is said to be \textit{solvable} if for some $m\geq 0$ we have $L^{(m)}=0$. It is clear that if $I$ is an ideal of $L$, then $I$ is also a Lie subalgebra of $L$. Thus, the terminology ``solvability" can be defined similarly for $I$.
	
		Every associative $K$-algebra $A$ gives rise to a Lie algebra by considering the same vector structure and the bracket given by $[a,b]=ab-ba$. The Lie algebra obtained in this way is called \textit{the Lie algebra associated to the associative algebra $A$}, and denoted by $A^-$. The associative algebra $A$ is called \textit{Lie solvable} if $A^-$ is solvable.  Under this convention, every (ring-theoretic) ideal of $A$ is an (Lie) ideal of $A^-$; but the converse does not hold. Thus, to distinguish the (Lie) ideals of $A^-$ and the (ring-theoretic) ideals of $A$, we use the term “Lie ideal of $A$” to mean some (Lie) ideal of the Lie algebra $A^-$, while the term “ideal” is still used to indicate some ring-theoretic ideal of the associative algebra $A$. Recall that a Lie algebra $L$ is called \textit{locally finite} if every finite set of elements of $L$ is contained in a finite-dimensional Lie subalgebra of $L$. We also recall that an (associative) $K$-algebra $A$ is called a \textit{locally finite $K$-algebra} if every finite set of elements of $A$ is contained in a finite-dimensional $K$-subalgebra of $A$. It is clear that if $A$ is locally finite as a $K$-algebra, then $A^-$ is a locally finite Lie algebra over $K$. In this paper, we particularly consider the special case where $A = L_K(E)$.

	\section{Leavitt path algebras in which every Lie ideal is an ideal}
		In this section, we prove the main result of the current paper, which classifies all Leavitt path algebras in which all Lie ideals are (ring-theoretic) ideals. It turns out that there are exactly three distinct classes of such Leavitt path algebras, each built on a specific type of graph (see Theorem \ref{theorem_3.8}). One of these classes is the class of commutator Leavitt path algebras, studied by Z. Mesyan in \cite{Pa_mesyan-13}. To establish the main theorem, we introduce the notion of a \textit{$p$-commutator graph}, where $p$ is a prime number. We fix the following notations. Let $E$ be a graph.  Let $u,v \in E^0$; we denote by $d(u,v)$ the length of the shortest path $p$ satisfying $s(p)=u$ and $r(p)=v$. If there is no such a path $p$, we set $d(u,v) = \infty$. For each $u \in E^0$ and $m \in \mathbb{N}$, we define $D(u,m) = \{v \in E^0 | d(u,v) \leq m\}$.
		\begin{definition}\label{definition_3.1}
			Let $E$ be a graph and $p$ a prime number. We say that $E$ is a \textit{$p$-commutator graph} if $E$ satisfies the following conditions:
			\begin{enumerate}[font=\normalfont]
				\item $E$ is acyclic.
				\item $E^0$ contains only regular vertices.
				\item For each $u\in E^0$, there is an $m\in\mathbb{N}$ such that for all $w\in E^0$ satisfying $d(u,w)=m+1$, the number of paths $q=e_1\dots e_k\in{\rm Path}(E)$ such that $s(q)=u$ and $r(q)=w$, and $s(e_2),\dots,s(e_k)\in D(u,m)$ is a multiple of $p$.
			\end{enumerate}
		\end{definition}
		In \cite{Pa_mesyan-13}, Mesyan has classified all commutator Leavitt path algebras, i.e., the algebras $L_K(E)$ such that $[L_K(E),L_K(E)]=L_K(E)$. For the reader's convenience, we restate Mesyan’s theorem here.
		\begin{theorem}[{\cite[Theorem 24]{Pa_mesyan-13}}]\label{theorem_3.2}
			Let $E$ be a graph and $K$ a field. Then $L_K(E)$ is a commutator ring if and only if ${\rm char}(K)=p>0$ and $E$ is a $p$-commutator graph. 
		\end{theorem}
		
		In the same paper, Z. Mesyan also showed that commutator Leavitt path algebras have the further interesting property that every Lie ideal is an ideal. This motivates the natural question of classifying the Leavitt path algebras which have such a property. The answer to this question is provided by Theorem \ref{theorem_3.8}, whose proof is rather lengthy and will be preceded by a few lemmas. 
		\begin{lemma}\label{lemma_3.3}
			Let $K$ be a field, and $A$ a $K$-algebra. If $I$ is an ideal of $A$, then $[I,I]$ is a Lie ideal of $A$. 
		\end{lemma}
		
		\begin{lemma}\label{lemma_3.4}
			Let $K$ be a field, and $G$ an abelian group which is written additively. Let $A$ be a $G$-graded $K$-algebra. Then, as a set, the additive subgroup $[A, A]$ of $A$ is $G$-graded. 
		\end{lemma}
		\begin{proof}
			Take $a,b\in A$ and write $a=\sum_{\sigma\in G}a_{\sigma}, b=\sum_{\tau\in G}b_{\tau}$. Then, we have $ab =\sum_{\omega}c_{\omega}, \text{ where } c_{\omega}=\sum_{\sigma+\tau=\omega} a_{\sigma}b_{\tau}$ and $ba=\sum_{\omega}c'_{\omega}, \text{ where } c'_{\omega}=\sum_{\sigma+\tau=\omega}b_{\tau}a_{\sigma}$. It follows that 
			$$
			[a,b]=ab-ba=\sum_{\omega}(c_{\omega}-c'_{\omega}).
			$$
			This implies that  $c_{\omega}-c'_{\omega}=\sum_{\sigma+\tau=\omega}(a_{\sigma}b_{\tau}-b_{\tau}a_{\sigma}) \in [A,A]$ for each $\omega$. Now, let $x$ be an arbitrary element of $[A, A]$. Then, there exist $a_1, \dots, a_k$ and $b_1,\dots, b_k$ in $A$ such that $x=\sum_{i=1}^{k}[a_i,b_i]$. For each $i\in\{1,\dots,k\}$, by what we have seen, each homogenous component of each $[a_i,b_i]$ belongs to $[A,A]$. It follows that each homogenous component of $x$ is also in $[A,A]$. The lemma is proved.
		\end{proof}
		\begin{lemma}[{\cite[Corollary 1.5.12]{Bo_abrams-ara-molina-LPA}}]\label{lemma_3.5_}
			Let $E$ be a graph and $K$ a field. Put $\mathscr{A}=\{\lambda\nu^*\;|\; \lambda, \nu\in{\rm Path}(E) \text{ and }r(\lambda)=r(\nu)\}$. For each $v\in{\rm Reg}(E) $, let $\{e_1^v,\dots,e_{n_v}^v\}$ be an enumeration of the elements of $s^{-1}(v)$. Then, a $K$-basis of the Leavitt path algebra $L_K(E)$ is given by the family 
			$$\mathscr{B}=\mathscr{A}\backslash\{\lambda e_{n_v}^v(e_{n_v}^v)^*\nu^*: \lambda, \nu\in{\rm Path}(E), r(\lambda)=r(\nu)=v\in {\rm Reg}(E)\}.$$
		\end{lemma}
		\begin{lemma}\label{lemma_3.6}
			Let $E$ be a graph and $K$ a field. If $E$ contains an isolated loop, then $L_K(E)$ contains a Lie ideal which is not an ideal.
		\end{lemma}
		
		\begin{proof}
			It is clear that $E$ may be decomposed as $E= c \sqcup F$, where $F$ is a (possibly empty) subgraph of $E$. This decomposition implies that $L_K(E)=L_K(c)\oplus L_K(F)$. In $L_K(c)$, consider the $K$-subspace $Kc$ which is a Lie ideal as $L_K(c)$ is commutative. However, $Kc$ is not an ideal of $L_K(c)$ as $c^2\notin Kc$. It follows that $Kc$ is a Lie ideal of $L_K(E)$ which is not an ideal of $L_K(E)$.
		\end{proof}

		\begin{lemma}\label{lemma_3.7}
			Let $E$ be a graph and $K$ a field. If $E$ has more than one isolated vertex, then $L_K(E)$ contains a Lie ideal which is not an ideal.
		\end{lemma}
		\begin{proof}
			Let $u,v$ be two distinct isolated vertices in $E^0$. Then $E$ may be decomposed as $E= u \sqcup v \sqcup F$, where $F$ is a (possibly empty) subgraph of $E$. Then, the $K$-subspace $K(u+v)$ of $L_K(E)$ generated by $u+v$ is a Lie ideal of $L_K(E)$. We claim that $K(u+v)$ is not an ideal of $L_K(E)$. Indeed, if otherwise, we would have $u=(u+v)u\in K(u+v)$. This would imply that there exists $k\in K$ such that $u=k(u+v)$, and so $(k-1)u+kv=0$ for some $k\ne0$. Because the set $\{u,v\}$ is linearly independent over $K$, we have $k-1=k=0$, which is impossible. It follows that $K(u+v)$ is not an ideal of $L_K(E)$.
		\end{proof}
		
		\begin{remark}\label{remark_1}
			Let $E$ be a graph and $K$ a field. It is known that the Leavitt path algebra $L_K(E)$ is commutative if and only if $E$ is a disjoint union of isolated loops and isolated vertices. Thus, in view of Lemmas \ref{lemma_3.6} and \ref{lemma_3.7}, we easily obtain the following:  A commutative Leavitt path algebra $L_K(E)$ has the property that every Lie ideal is an ideal if and only if $E$ consists of precisely one vertex. 
		\end{remark}
		
		Now, we return our attention to non-commutative Leavitt path algebras. The following theorem, which determines the graph $E$ and the field $K$ for which $L_K(E)$ has the property that every Lie ideal is an ideal, is the main result of this section.
		\begin{theorem}\label{theorem_3.8}
			Let $E$ be a graph and $K$ a field. Assume that $L_K(E)$ is non-commutative. Then, $L_K(E)$ has the property that every Lie ideal is an ideal if and only if ${\rm char}(K)=p>0$ and $E$ satisfies one of the following conditions:
			\begin{enumerate}[font=\normalfont]
				\item $E$ is a $p$-commutator graph.
				\item  $E$ is a disjoint union of a $p$-commutator graph $F$ and an isolated vertex $v$. In this case, 
				$$
				L_K(E)=Kv\oplus L_K(F).
				$$
				\item  $E^0=H\sqcup \{v\}$, where $H$ is a non-empty hereditary and saturated subset such that the porcupine graph $P_H$ is a $p$-commutator graph, and $v$ is the unique infinite emitter and a source in $E$. In this case,
				$$
				L_K(E)=Kv+I(H).
				$$
			\end{enumerate}
			In particular, the graph $E$ is an acyclic graph.
		\end{theorem}
		\begin{proof}
			For a proof of the ``only if" part, we assume that every Lie ideal of $L_K(E)$ is an ideal. Let $\Gamma$ be the (possibly empty) set of all isolated vertices in $E$. By Lemma \ref{lemma_3.7}, $\Gamma$ contains at most one vertex $v$. It follows that either $E=F$ or $L_K(E)=L_K(F)\oplus Kv$ and in any case 
			$$
			[L_K(E),L_K(E)]=[L_K(F),L_K(F)].
			$$
			Then $I:=[L_K(F),L_K(F)]$ is a graded (by Lemma \ref{lemma_3.4}) Lie ideal and hence by hypothesis a graded ideal of $L_K(E)$. We divide our situation into three possible cases:
			
			\medskip 
			
			{\noindent \textbf{Case 1:}} $I=0$. Then, it is clear that $L_K(F)$ is commutative, from which it follows that $L_K(E)$ is commutative too, a contradiction. 
			
			\medskip 
			
			{\noindent \textbf{Case 2:}} $I=L_K(F)$. In this case, $L_K(F)$ is a commutator ring and so,  by Theorem \ref{theorem_3.2}, ${\rm char} (K)=p>0$ and $F$ is a $p$-commutator graph. Either (1) or (2) holds in this case. 
			
			\medskip 
			
			{\noindent \textbf{Case 3:}} $0\ne I\ne L_K(F)$. Put $H=E^0\cap I$ and $S=\{v\in B_H\;|\; v^H\in I\}\subseteq B_H$. Then, it is known that $I=I(H,S)$. Since $I \subseteq L_K(F)$, we conclude that $I$ is also a graded ideal of $L_K(F)$ and $H=F^0\cap I$, and $S\subseteq F^0$. In other words, the pair $(H,S)$ is an admissible pair of $F$.	Consider the quotient graph $F\backslash (H,S)$ with 
			\begin{align*}
				&	 (F\backslash (H,S))^0=F^0\backslash H\cup\{v'\in B_H\backslash S\}, \\
				&	(F\backslash (H,S))^1=\{e\in F^1\;|\; r(e)\notin H\} \cup \{e'\;|\;e\in F^1, r(e)\in B_H\backslash S\},
			\end{align*}
			and $r$ and $s$ are extended to $(F\backslash(H,S))^1$ by setting $s(e')=s(e)$ and $r(e')=r(e)'$. As $I$ is a graded ideal of $L_K(F)$, by \cite[Theorem 2.4.15]{Bo_abrams-ara-molina-LPA}, we have  $L_K(F)/I\cong L_K(F\backslash (H,S)),$ which is a commutative ring; indeed $I=[L_K(F), L_K(F)]$. Because $L_K(F\backslash (H,S))$ is a homomorphic image of $L_K(F)$, every Lie ideal of $L_K(F\backslash (H,S))$ is also an ideal. Thus, Remark \ref{remark_1} indicates that $F\backslash (H,S)$ consists precisely of only one vertex, say $v$. Therefore, we have $(F\backslash (H,S))^0=F^0\backslash H=\{v\}$ and $(F\backslash (H,S))^1=\varnothing$; consequently, $B_H=\varnothing$: indeed, if $u\in B_H$ there exists an edge $e$ with $s(e)=u$ and $r(e)\notin H$.
			
			We claim that $v$ is both an infinite emitter and a source in $E$. Since $v\notin H$ and $(F\backslash(H,S))^1=\varnothing$, it is $r_E^{-1}(v)=r_F^{-1}(v)=\varnothing$; moreover since $v\in F^0$ and no vertex in $F^0$ is isolated in $E$, $s^{-1}(v)\ne\varnothing$. Therefore $v$ is a source. If $v$ is a regular vertex of $E$, then, as $H$ is saturated, we would have $v\in H$, a contradiction. It follows that $v$ is an infinite emitter in $E$ with $r(s^{-1}(v))\subseteq H$. The claim is proved.
			
			Let $P_{H}$ be the porcupine graph defined by the admissible pair $(H,\varnothing)$ (see \cite[Definition 3.1]{Pa_vas-22}). By \cite[Theorem 3.3]{Pa_vas-22}, we know that $I$ is graded isomorphic to the Leavitt path algebra $L_K(P_{H})$. In what follows, we shall prove that ${\rm char}(K)=p>0$ and $P_H$ is a $p$-commutator graph. In order to obtain this, let us make the construction of $P_H$ more detailed. Using the notations \cite[Definition 3.1]{Pa_vas-22}, a simple calculation shows that $F_1(H,\varnothing)=\{e\in E^1\;|\;s(e)\in F\backslash H \text{ and } r(e)\in H\}=s^{-1}(v)$ which is an infinite set as $v$ is an infinite emitter. Thus,
			\begin{align*}
				&	(P_H)^0= H\cup \{w^e\;|\; e\in s^{-1}(v)\}, \\
				&	(P_H)^1=\{e\in E^1\;|\; s(e)\in H\}\cup \{f^e\;|\; e\in s^{-1}(v)\}.
			\end{align*}
			The $\mathbf{r}$ and $\mathbf{s}$ maps of $P_H$ are the same as in $E$ for the common edges and $\mathbf{s}(f^e)=w^e$ and $\mathbf{r}(f^e)=r(e)$ for all $e\in F_1(H,\varnothing)$. Because all new vertices $w^e$ in $P_H$ are regular, every infinite emitter of $P_H$ is also an infinite emitter of $E$. Put $J=[L_K(P_H),L_K(P_H)]\cong [I,I]$ which is a graded ideal of $L_K(E)$ by Lemmas \ref{lemma_3.3} and \ref{lemma_3.4}. It follows that $J$ is a graded ideal of $L_K(P_H)$. If $J=0$, then $L_K(P_H)$ is commutative; thus, in view of Lemma \ref{lemma_3.7},  $P_H$ is precisely only one vertex. This contradicts to the fact that $(P_H)^0$ is infinite. Therefore, we may assume that $J\ne0$. We consider two subcases in turn:
			
			\smallskip 
			
			{\noindent \textbf{Case 3.1:}} $J \ne L_K(P_H)$. Put $H_1=J\cap (P_H)^0$. Repeat the arguments as we did for $I$, $L_K(E)$ and $H$, we conclude that there exists a vertex  $v'\in (P_H)^0\backslash H_1$ which is both a source and an infinite emitter in $P_H$ with 
			$$
			\mathbf{r}(\mathbf{s}^{-1}(v'))\subseteq H_1\subseteq (P_H)^0=H\cup \{w^e\;|\; e\in s^{-1}(v)\}.
			$$
			As each new vertex $w^e$ is a source and regular, it follows that $\mathbf{r}(\mathbf{s}^{-1}(v'))\subseteq H$ and $v'\in H\subseteq E^0$. In $E$, the set  $H'=H\backslash \{v'\}$ is hereditary and saturated and $B_{H'}=\varnothing$. Thus, the quotient graph $E\backslash H'=\{v'\}\sqcup \Gamma$, which consists of isolated vertices. This implies that $L_K(E)/I(H')\cong L_K(E\backslash H')$ is commutative; and, consequently, $[L_K(E), L_K(E)]=I(H)\subseteq I(H')$, a contradiction. Therefore, this subcase cannot occur. 
			
			\smallskip 
			
			{\noindent \textbf{Case 3.2:}} $J = L_K(P_H)$. This means that ${\rm char }(K)=p >0$ and  $P_H$ is a $p$-commutator graph. Consequently, the graph $P_H$ contains neither cycles nor infinite emitters. It follows that $E$ is acyclic and $v$ is the unique infinite emitter in $E$. Moreover, the construction of $F$ shows that $L_K(F)=Kv+I(H)$. If $\Gamma=\varnothing$, then $E=F$, and so (3) holds, we are done. To complete the proof of the ``only if" part, we shall prove that the case $|\Gamma|=1$ cannot occur. So, assume to the contrary that $\Gamma=\{w\}$. Then,
			$$
			L_K(E)=Kw\oplus \left( Kv+I(H)\right).
			$$
			Let $M$ be the ideal of $L_K(E)$ generated by $r(s^{-1}(v))$. Put $H_M=E^0\cap M$.  Consider the $K$-subspace $U=K(w+v)+M$ of $L_K(E)$. We claim that $U$ is a Lie ideal of $L_K(E)$. Indeed, for $\alpha, \beta\in{\rm Path}(E)$ with $r(\alpha)=r(\beta)$, and $x:=k(w+v)+a\in U$, where $k\in K$ and $a\in M$, we have
			$$
			[x,\alpha\beta^*]=[kw,\alpha\beta^*]+[kv,\alpha\beta^*]+[a,\alpha\beta^*]=[kv,\alpha\beta^*]+[a,\alpha\beta^*],
			$$
			as $w$ is an isolated vertex in $E$. Because $M$ is an ideal of $L_K(E)$, we conclude that $[a,\alpha\beta^*]\in U$. This assures that $[x,\alpha\beta^*]\in U$ if and only if $[v,\alpha\beta^*]\in U$. A simple calculation shows that
			\begin{align*}
				&	[v,\alpha\beta^*]=v\alpha\beta^*-\alpha\beta^*v =    \begin{cases} 
					\alpha\beta^*                                  & \text{ if } s(\alpha) = v\text{ and } s(\beta)\ne v,  \\
					-\alpha\beta^*                                & \text{ if } s(\alpha)\ne v\text{ and } s(\beta) = v,  \\
					0             											 & \text{ otherwise.}
				\end{cases}
			\end{align*}
			Now, let us consider the above cases in turn. First, we consider the case where $[v,\alpha\beta^*]=\alpha\beta^*$ with $s(\alpha)=v$ and $s(\beta)\ne v$. Since $H_M$ is hereditary, we get that $r(\alpha)\in H_M$, which shows that $[v,\alpha\beta^*]=\alpha\beta^*=\alpha r(\alpha)\beta^*\in M\subseteq U$.
			Similarly, we also see that $[v,\alpha\beta^*]\in U$ in the second case. Finally, in the last case, we have $[v,\alpha\beta^*]=0\in U$. Therefore $[x,\alpha\beta^*]\in U$ in any cases, proving that $U$ is a Lie ideal of $L_K(E)$. The claim is proved. Next, we prove that $U$ is not an ideal of $L_K(E)$. Once again, as $E$ is acyclic, it follows that $M$ is graded and generated by $H_M$. By the construction of $F$, it is clear that $w,v\notin H_M$. For each $u\in{\rm Reg}(E) $, let $\{e_1^u,\dots,e_{n_u}^u\}$ be an enumeration of the elements of $s^{-1}(u)$. In view of \cite[Lemma 2.4.1]{Bo_abrams-ara-molina-LPA} and Lemma \ref{lemma_3.5_}, we conclude that the set
			$$
			\mathcal{B}=\{\lambda\nu^*:r(\lambda)=r(\nu)\in H_M\}\backslash\{\lambda e_{n_u}^v(e_{n_u}^u)^*\nu^*: r(\lambda)=r(\nu)=u\in {\rm Reg}(E)\}
			$$
			is a $K$-basis for $M$. Being a subset of a $K$-basis for $L_K(E)$, the set $\{w, v\}\cup\mathcal{B}$ is linearly independent. Assume by contradiction that $U$ is an ideal of $L_K(E)$. It follows that $v=v(w+v)\in U=K(w+v)+M$. This means there exist $k, k_1,\dots, k_n\in K$ and $\lambda_1\nu_1^*, \dots, \lambda_n\nu_n^*\in \mathcal{B}$ such that 
			$$
			v=k(w+v)+k_1\lambda_1\nu_1^*+\cdots+k_n\lambda_n\nu_n^*,
			$$
			which is impossible as  $\{w, v\}\cup\mathcal{B}$ is linearly independent. Therefore, $U$ cannot be an ideal of $L_K(E)$. This violates the hypothesis that every Lie ideal of $L_K(E)$ is an ideal. This implies that the case $|\Gamma|=1$ cannot occur.
			
			Now, we give the proof of the converse of the theorem. If $E$ satisfies the condition (1), then the converse is exactly \cite[Theorem 27]{Pa_mesyan-13}. Assume that $E$ satisfies the condition (2). Let $U$ be a Lie ideal of $L_K(E)$. Then, there exist $K$-subspaces $U_1$ and $U_2$ of $Kw$ and $L_K(F)$ respectively such that $U=U_1\oplus U_2$. A straightforward checking shows that $U_2$ is a Lie ideal of $L_K(F)$. As $F$ is a $p$-commutator graph, we get that $L_K(F)$ is a commutator ring. It follows from (1) that $U_2$ is an ideal of $L_K(F)$. Moreover, because $Kw$ is commutative, the subspace $U_1$ is certainly an ideal of $Kw$. It follows that $U$ is an ideal of $L_K(E)$.
			
			Now, assume that $E$ satisfies the condition (3). It is clear that $B_H=\varnothing$. Let $P_H$ be the porcupine graph determined by the admissible pair $(H,\varnothing)$. The sets of vertices and edges of the Porcupine graph $P_H$ are as follows.
			\begin{align*}
				&	(P_H)^0= H\cup \{w^e\;|\; e\in s^{-1}(v)\}, \\
				&	(P_H)^1=\{e\in E^1\;|\; s(e)\in H\}\cup \{f^e\;|\; e\in s^{-1}(v)\}.
			\end{align*}
			The maps $\mathbf{r}$ and $\mathbf{s}$ of $P_H$ are the same as in $E$ for the common edges and $\mathbf{s}(f^e)=w^e$ and $\mathbf{r}(f^e)=r(e)$ for all $e\in s^{-1}(v)$. Let $E_H$ be the subgraph of $E$ with $(E_H)^0=H$ and $(E_H)^1=\{e\in E^1\;|\; s(e)\in H\}$. Then, $E_H$ is also a subgraph of $P_H$, and hence, it is also a $p$-commutator graph. Moreover, as $E$ contains no cycles, if $\mu\in {\rm Path}(E)$ such that $s(\mu)\in H$ then $\mu\in{\rm Path}(E_H)$ and $r(\mu)\in H$. This fact will be used regularly without referring again. The proof follows from some modifications of the proof of \cite[Lemma 26 and Theorem 27]{Pa_mesyan-13} and finished in four steps.
			
			\medskip 
			
			{\noindent \textbf{Step 1:}} Let $u,u'\in E^0$, and $U$ be the Lie ideal of $L_K(E)$ generated by $u$. We claim that if there exists $t\in {\rm Path}(E)\backslash E^0$ such that $s(t)=u$ and $r(t)=u'$, then $u'\in U$. Since $r(t)\in H$ and $E_H$ is a $p$-commutator graph, one can repeat the arguments of \cite[Theorem 27(1)]{Pa_mesyan-13}.
			
			\medskip 
			
			{\noindent \textbf{Step 2:}} Let $u\in E^0$, and $U\subseteq L_K(E)$ be the Lie ideal generated by $u$. Then the same arguments as in the proof of the assertion (2) of \cite[Lemma 26]{Pa_mesyan-13} shows that for all $t,q\in{\rm Path}(E)$ such that $s(t)=u$, $s(q)=u$ or $r(t)=r(q)=u$, we have $tq^*\in U$.
		
			\medskip 
			
			{\noindent \textbf{Step 3:}} Let $\mu=\sum_{i=1}^{n}a_it_iq_i^*\in L_K(E)\backslash\{0\}$ be an arbitrary element, where $a_i\in K\backslash\{0\}$ and $t_i, q_i\in{\rm Path}(E)$ with $r(t_i)=r(q_i)$ for all $i\in\{1,\dots,n\}$. Let $U\subseteq L_K(E)$ be the Lie ideal generated by $\mu$.
			Since $v$ is a source, if $s(t_{i_0})=v$ then $\mu=a_{i_0}v+\sum_{i\ne i_0}a_it_iq_i^*$ with $r(t_i)\ne v$ for each $i\ne i_0$. We claim  that $r(t_i)\in U$ for all $i\in\{1,\dots,n\}$. Indeed, if $\mu = av$ for some $a\in K\backslash\{0\}$, then the conclusion clearly holds. So, we may assume that $\mu \ne av$ for any $a\in K$. In this case, we have $r(t_i)\in H= (E_H)^0$. By the same arguments as in the proof of (3) of \cite[Lemma 26]{Pa_mesyan-13} applied for $E_H$, we obtain that $r(t_i)\in U$. Since $r(t_i)\in H$ for each $i\ne i_0$, by Step 2, we get that $t_iq_i^*\in U$, from which it follows that $\sum_{i\ne i_0}t_iq_i^*\in U$. As $\mu\in U$ and $\sum_{i\ne i_0}t_iq_i^*\in U$, we get that $a_{i_0}v\in U$, and hence $v\in U$. Therefore, $r(t_i)\in U$ for all $i\in\{1,\dots,n$\}.
			
			\medskip 
			
			{\noindent \textbf{Step 4:}} We claim that every Lie ideal of $L_K(E)$ is also an ideal. The proof is completely the same as that of \cite[Theorem 27]{Pa_mesyan-13}, so we omit it. Thus, the converse of the theorem is established. 
		\end{proof}
		Let us give some examples of the graphs described in (1), (2) and (3) of Theorem \ref{theorem_3.8}. Examples of graphs satisfying (1) are those given by Z. Mesyan in  \cite[Proposition 31 and Example 32]{Pa_mesyan-13}. For convenience of the reader, we still present them in the next two examples.
		\begin{example}[{\cite[Proposition 31]{Pa_mesyan-13}}]\label{example_1}
			Let $K$ be a field with ${\rm char}(K)=p>0$. For each integer $n\geq 1$, let $E_n$ be the following graph:
			
			
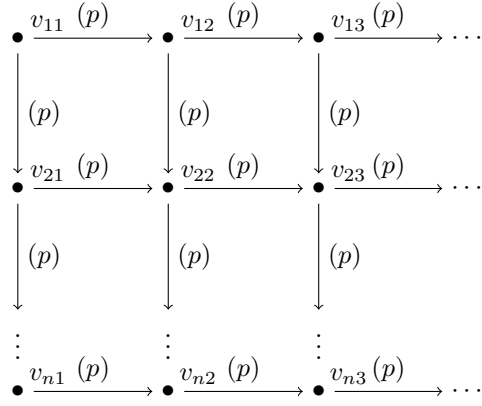
\begin{center}
				\begin{tikzpicture}
					% row 1
					\node at (0,0) (A11) {$\bullet$};
					\node at (0.4,0.2) (v11) {$v_{11}$};
					\node at (2,0) (A12) {$\bullet$};
					\node at (2.4,0.2) (v12) {$v_{12}$};
					\node at (4,0) (A13) {$\bullet$};
					\node at (4.4,0.2) (v13) {$v_{13}$};
					\node at (6,0) (A14) {$ \cdots$};
					
					\draw [->] (A11) to node[above] {$(p)$} (A12);
					\draw [->] (A12) to node[above] {$(p)$} (A13);
					\draw [->] (A13) to node[above] {$(p)$} (A14);
					
					% row 2
					\node at (0,-2) (A21) {$\bullet$};
					\node at (0.4,-1.8) (v21) {$v_{21}$};
					\node at (2,-2) (A22) {$\bullet$};
					\node at (2.4,-1.8) (v22) {$v_{22}$};
					\node at (4,-2) (A23) {$\bullet$};
					\node at (4.4,-1.8) (v23) {$v_{23}$};
					\node at (6,-2) (A24) {$ \cdots$};
					
					\draw [->] (A21) to node[above] {$(p)$} (A22);
					\draw [->] (A22) to node[above] {$(p)$} (A23);
					\draw [->] (A23) to node[above] {$(p)$} (A24);
					
					% row 3
					\node at (0,-4) (A31) {$\vdots$};
					\node at (2,-4) (A32) {$\vdots$};
					\node at (4,-4) (A33) {$\vdots$};
					
					% row 4
					\node at (0,-4.7) (An1) {$\bullet$};
					\node at (0.4,-4.5) (vn1) {$v_{n1}$};
					\node at (2,-4.7) (An2) {$\bullet$};
					\node at (2.4,-4.5) (vn2) {$v_{n2}$};
					\node at (4,-4.7) (An3) {$\bullet$};
					\node at (4.4,-4.5) (vn3) {$v_{n3}$};
					\node at (6,-4.7) (An4) {$ \cdots$};
					
					\draw [->] (An1) to node[above] {$(p)$} (An2);
					\draw [->] (An2) to node[above] {$(p)$} (An3);
					\draw [->] (An3) to node[above] {$(p)$} (An4);

					% column 1
					\draw [->] (A11) to node[right] {$(p)$} (A21);
					\draw [->] (A21) to node[right] {$(p)$} (A31);
					
					% column 2
					\draw [->] (A12) to node[right] {$(p)$} (A22);
					\draw [->] (A22) to node[right] {$(p)$} (A32);
					
					% column 4
					\draw [->] (A13) to node[right] {$(p)$} (A23);
					\draw [->] (A23) to node[right] {$(p)$} (A33);
				\end{tikzpicture}
			\end{center}
			The notation $(p)$ in the above graph indicates that there are $p$ edges from a vertex to another.  Then, for each $n\geq1$, $E_n$ is a $p$-commutator graph. For each $k\in\{1,\dots,n\}$, it can be seen that $H_k:=\bigcup_{i=k}^n(\bigcup_{j=1}^{\infty}\{v_{ij}\})$ is a hereditary and saturated subset of $E_n^0$. Therefore, $E^0\supseteq H_1\supseteq H_2\supseteq\dots\supseteq H_n$ are precisely hereditary and saturated subsets of $E^0$. Consequently, $L_K(E_n)$ has exactly $n+1$ ideals, which form a chain. In particular, we have $L_K(E_n)\not\cong L_K(E_m)$ if $n\ne m$.
		\end{example}
		
		\begin{example}[{\cite[Example 32]{Pa_mesyan-13}}]\label{example_2}
			Let $K$ be a field with ${\rm char}(K)=p>0$, and $E_{\infty}$ be the graph pictured below.
			
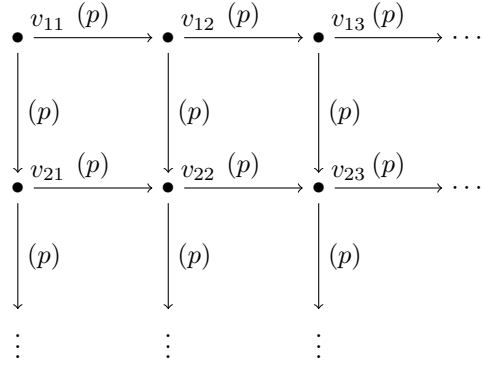
\begin{center}
				\begin{tikzpicture}
					% row 1
					\node at (0,0) (A11) {$\bullet$};
					\node at (0.4,0.2) (v11) {$v_{11}$};
					\node at (2,0) (A12) {$\bullet$};
					\node at (2.4,0.2) (v12) {$v_{12}$};
					\node at (4,0) (A13) {$\bullet$};
					\node at (4.4,0.2) (v13) {$v_{13}$};
					\node at (6,0) (A14) {$ \cdots$};
					
					\draw [->] (A11) to node[above] {$(p)$} (A12);
					\draw [->] (A12) to node[above] {$(p)$} (A13);
					\draw [->] (A13) to node[above] {$(p)$} (A14);
					
					% row 2
					\node at (0,-2) (A21) {$\bullet$};
					\node at (0.4,-1.8) (v21) {$v_{21}$};
					\node at (2,-2) (A22) {$\bullet$};
					\node at (2.4,-1.8) (v22) {$v_{22}$};
					\node at (4,-2) (A23) {$\bullet$};
					\node at (4.4,-1.8) (v23) {$v_{23}$};
					\node at (6,-2) (A24) {$ \cdots$};
					
					\draw [->] (A21) to node[above] {$(p)$} (A22);
					\draw [->] (A22) to node[above] {$(p)$} (A23);
					\draw [->] (A23) to node[above] {$(p)$} (A24);
				
					% row 4
					\node at (0,-4) (A31) {$\vdots$};
					\node at (2,-4) (A32) {$\vdots$};
					\node at (4,-4) (A33) {$\vdots$};
					
					% column 1
					\draw [->] (A11) to node[right] {$(p)$} (A21);
					\draw [->] (A21) to node[right] {$(p)$} (A31);
					
					% column 2
					\draw [->] (A12) to node[right] {$(p)$} (A22);
					\draw [->] (A22) to node[right] {$(p)$} (A32);
					
					% column 4
					\draw [->] (A13) to node[right] {$(p)$} (A23);
					\draw [->] (A23) to node[right] {$(p)$} (A33);
				\end{tikzpicture}
			\end{center}
			Then, $E_{\infty}$ is a $p$-commutator graph, and $L_K(E_{\infty})\not\cong L_K(E_n)$ for any $n$. It can be seen that the hereditary and saturated subsets of $E_{\infty}^0$ are of the form $\{v_{ij}:|\;i\geq k, j\geq l\}$, where $k,l\geq1$. Therefore, $L_K(E_{\infty})$ has infinitely many ideals that do not form a chain.
		\end{example}
		
		\begin{example}
			One can obtain an example of a graph satisfying the assertion (2) of Theorem \ref{theorem_3.8} by adding an isolated vertex to one of the graphs given in Examples \ref{example_1} and \ref{example_2}.
		\end{example}
		Finally, we give an example of a graph satisfying (3).
		\begin{example}\label{exmaple_4}
			Let $K$ be a field with ${\rm char}(K)=p>0$, and $E$ be the following graph:			
			
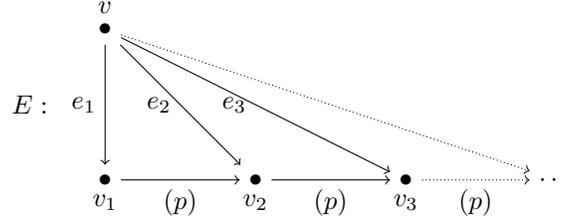
\begin{center}
				\begin{tikzpicture}
					\node at (-1,-1) (E) {$E:$};
					\node at (0,0) (O) {$\bullet$};
					\node at (0,0.3) (v) {$v$};
					\node at (0,-2) (A) {$\bullet$};
					\node at (0,-2.3) (v1) {$v_1$};				
					\node at (2,-2) (B) {$\bullet$};
					\node at (2,-2.3) (v2) {$v_2$};					
					\node at (4,-2) (C) {$\bullet$};
					\node at (4,-2.3) (v3) {$v_3$};						
					\node at (6,-2) (dots) {$\cdots$};
					
					\draw [->] (O) to node[left] {$e_1$} (A);
					\draw [->] (O) to node[left] {$e_2$} (B);
					\draw [->] (O) to node[left] {$e_3$} (C);
					\draw [->,densely dotted] (O) to (dots);
					\draw [->] (A) to node[below] {$(p)$} (B);
					\draw [->] (B) to node[below] {$(p)$} (C);
					\draw [->,densely dotted] (C) to node[below] {$(p)$} (dots);
				\end{tikzpicture}
			\end{center}
			Put $H = \{v_1, v_2, \dots\}$. Then,  $E^0=H\sqcup \{v\}$  and $H$ is a hereditary and saturated subset of $E^0$. The porcupine graph $P_H$ is pictured below.
			
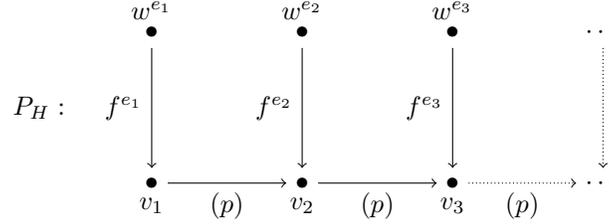
\begin{center}
				\begin{tikzpicture}
					% row 1
					\node at (-1.5,-1) (PH) {$P_H:$};
					\node at (0,0) (M) {$\bullet$};
					\node at (0,0.3) (we1) {$w^{e_1}$};
					\node at (2,0) (N) {$\bullet$};
					\node at (2,0.3) (we2) {$w^{e_2}$};
					\node at (4,0) (P) {$\bullet$};
					\node at (4,0.3) (we3) {$w^{e_3}$};
					\node at (6,0) (dots-above) {$\cdots$};
									
					% row 2
					\node at (0,-2) (A) {$\bullet$};
					\node at (0,-2.3) (v1) {$v_1$};				
					\node at (2,-2) (B) {$\bullet$};
					\node at (2,-2.3) (v2) {$v_2$};					
					\node at (4,-2) (C) {$\bullet$};
					\node at (4,-2.3) (v3) {$v_3$};						
					\node at (6,-2) (dots) {$\cdots$};
					
					\draw [->] (M) to node[left] {$f^{e_1}$} (A);
					\draw [->] (N) to node[left] {$f^{e_2}$} (B);
					\draw [->] (P) to node[left] {$f^{e_3}$} (C);
					\draw [->,densely dotted] (dots-above) to (dots);

					\draw [->] (A) to node[below] {$(p)$} (B);
					\draw [->] (B) to node[below] {$(p)$} (C);
					\draw [->,densely dotted] (C) to node[below] {$(p)$} (dots);
				\end{tikzpicture}
			\end{center}
			 One can easily check that $P_H$ is a $p$-commutator graph, and that $E$ satisfies all requirements in (3) of Theorem \ref{theorem_3.8}. It follows that $L_K(E)$ has the property that every Lie ideal is an ideal. In $E$, there are exactly three hereditary and saturated subsets: $\varnothing$, $H$ and $E^0$. Therefore, all Lie ideals in $L_K(E)$ are $0$, $I(H)$ and $L_K(E)$ itself.
		\end{example}
	\section{A class of Lie semisimple Leavitt path algebras}
		The current section gives an application of Theorem \ref{theorem_3.8}. More precisely, we show that the Lie algebras $L_K(E)^-$ of the Leavitt path algebras $L_K(E)$ obtained in Theorem \ref{theorem_3.8} are locally finite, infinite-dimensional Lie algebras, whose certain radicals can be calculated explicitly. Also, we will identify whether these Lie algebras are semisimple. First, we find the condition for $L_K(E)$ to be locally finite as a $K$-algebra. To obtain this, we need to consider a certain subgraph of $E$ constructed by a given set of regular vertices (see \cite[Definition 1.5.16]{Bo_abrams-ara-molina-LPA}).
		\begin{definition}
			Let $E$ be a graph. For a set of regular vertices $X\subseteq {\rm Reg}(E)$, we set $Y={\rm Reg}(E)\backslash X$. Let $Y'=\{v'\;|\; v\in Y\}$ be a disjoint copy of $Y$. For $v\in Y$ and for each edge $e\in r_E^{-1}(v)$, we consider a new symbol $e'$. We define the graph $E(X)$ as follows:
			$$
			E(X)^0=E^0\sqcup Y' \text{ and } E(X)^1=E^1\sqcup \{e'\;|\; r_E(e)\in Y\}.
			$$
			For each $e\in E^1$, we define $r_{E(X)}(e)=r_E(e)$ and $s_{E(X)}(e)=s_E(e)$, and define $s_{E(X)}(e')=s_E(e)$ and $r_{E(X)}(e')=r_E(e)'$.
		\end{definition}
		We also use the notion of a complete subgraph which was introduced in \cite[Definition 1.5.16]{Bo_abrams-ara-molina-LPA}. For convenience, we give here the definition of such a subgraph, and we refer the readers to \cite{Bo_abrams-ara-molina-LPA} for more detail.
		\begin{definition}
			A subgraph $F$ of a graph $E$ is said to be \textit{complete} if for each $v\in F^0$, whenever $s_F^{-1}(v)\ne\varnothing$ and $0<|s_E^{-1}(v)|<\infty$, then $s_F^{-1}(v)=s_E^{-1}(v)$.
		\end{definition}
		The next lemma plays an important role in finding the condition on the graph $E$ for which $L_K(E)$ is locally finite as a $K$-algebra.
		\begin{lemma}\label{lemma_4.3}
			Let $E$ be a graph and $K$ a field. Then, every finite subgraph $F$ of $E$ is always contained in a finite graph, say $\dot{F}$, such that there is an injective $K$-algebra homomorphism $L_K(\dot{F})\to L_K(E)$.
		\end{lemma}
		\begin{proof}
			Put $X=F^0\cap {\rm Reg}(E)$. Let $\bar{F}$ be a subgraph of $E$ constructed as follows. The set of vertices and edges of $\bar{F}$ are 
			$$
			(\bar{F})^0=F^0\cup \{r_E(e): e\in s_E^{-1}(X)\}\text{ and }(\bar{F})^1=F^1\cup \{e\in s_E^{-1}(X)\}.
			$$
			In words, $\bar{F}$ is obtained by adding to $F$ all edges in $E^1$,  whose sources are regular vertices of $E^0$ contained in $F^0$, together with their ranges. Since $X$ contains only a finite number of regular vertices, we obtain that both sets $\{r_E(e): e\in s_E^{-1}(X)\}$ and $\{e\in s_E^{-1}(X)\}$ are finite. It follows that $\bar{F}$ is a finite graph. Next, we show that $\bar{F}$ is complete. To do this, assume that $v\in \bar{F}^0$ with $s_{\bar{F}}^{-1}(v)\ne\varnothing$ and $0<|s_E^{-1}(v)|<\infty$. This implies that $v\in X$, and so $s_E^{-1}(v)\subseteq s_E^{-1}(X)\subseteq (\bar{F})^1$. It follows that $s_{\bar{F}}^{-1}(v)=s_E^{-1}(v)$, and hence $\bar{F}$ is complete. Moreover, by \cite[Theorems 1.5.18 and 1.6.10]{Bo_abrams-ara-molina-LPA}, we get an embedding
			$$
			L_K(\bar{F}({\rm Reg}(\bar{F})\cap{\rm Reg}(E)))\cong C_K^{{\rm Reg}(\bar{F})\cap{\rm Reg}(E)}(\bar{F})\to L_K(E).
			$$
			Now, put $\dot F=\bar{F}({\rm Reg}(\bar{F})\cap{\rm Reg}(E))$ which is a finite graph by \cite[Propostion 1.5.21]{Bo_abrams-ara-molina-LPA}. The lemma is proved.
		\end{proof}
		
		The following proposition gives a necessary and sufficient condition for Leavitt path algebras to be locally finite as an algebra. We note in passing that the meaning of the term “locally finite” in the next proposition is different from that used in \cite{Pa_abrams-pino-molina-08}. 
		\begin{proposition}\label{proposition_4.4}
			Let  $E$ be a graph and $K$ a field. Then $L_K(E)$ is a locally finite $K$-algebra if and only if $E$ is acyclic.
		\end{proposition}
		\begin{proof}
			It is known that $L_K(E)$ is spanned as a vector space over $K$ by the set of monomials $\{\gamma\lambda^*\;|\; \gamma\lambda\in {\rm Path}(E) \text{ such that } r(\gamma)=r(\lambda)\}$. Therefore, to give a proof of the ``if part", it is enough to prove that for any finite number of paths $\mu_i=e_{1i}\dots e_{n_i}\in{\rm Path}(E)$, for $1\leq i\leq k$, there always exists a finite-dimensional subalgebra of $L_K(E)$ containing $\mu_i$ and $\mu_i^*$ for all $1\leq i\leq k$. Let $F$ be the subgraph of $E$ with 
			\begin{align*}
				&	 F^0=\{s(e_{1i}), r(e_{1i}), \dots, r(e_{n_ii}) \text{ for all } 1\leq i\leq k\}, \\
				&	F^1=\{e_{1i},\dots,e_{n_ii} \text{ for all } 1\leq i\leq k \}.
			\end{align*}
			Then $F$ is finite, and so, by Lemma \ref{lemma_4.3}, $F$ is contained in the finite graph $\dot{F}$. Moreover, there is an embedding $L_K(\dot{F})\to L_K(E)$. Since $E$ is acyclic, by \cite[Proposition 1.5.21]{Bo_abrams-ara-molina-LPA}, the graph $\dot{F}$ is acyclic too. In view of \cite[Theorem 2.6.17]{Bo_abrams-ara-molina-LPA}, we conclude that $L_K(\dot{F})$ is finite-dimensional over $K$. Let $A$ be the $K$-subalgebra of $L_K(E)$ generated by $\mu_i$ and $\mu_i^*$. Since $\mu_i,\mu_i^*\in L_K(\dot{F})$ for all $i\in\{1,\dots,k\}$ and $L_K(\dot{F})$ is embedded $L_K(E)$, we get that $A$ is contained in $L_K(\dot{F})$. It follows that $A$ is finite-dimensional over $K$. 
			
			For the converse, assume that $L_K(E)$ is locally finite. If $E$ is not acyclic, then $E$ contain a cycle, say $c$. But then, by Lemma \ref{lemma_3.5_}, we have that $\{c^n\}_{n\in\mathbb{N}}$ is a linearly independent set contained in the subalgebra of $L_K(E)$ generated by $c$, a contradiction. This assures that $E$ is an acyclic graph.
		\end{proof}
		
		\begin{corollary}\label{corollary_4.5}
			Let  $E$ be a graph and $K$ a field. Assume that $L_K(E)$ has the property that every Lie ideal is an ideal. Then $L_K(E)$ is a locally finite $K$-algebra.  In particular, the Lie algebra $L_K(E)^-$ is locally finite as a Lie algebra over $K$.
		\end{corollary}
		\begin{proof}
			According to Theorem \ref{theorem_3.8}, it follows that $E$ is acyclic. Thus, the first assertion immediately follows from Proposition \ref{proposition_4.4}. 
		\end{proof}
		Recall that the \textit{solvable radical} of a finite-dimensional Lie algebra is defined to be its largest solvable ideal which is the sum of all solvable ideals. Such a radical may not exist in the infinite-dimensional case.  However, there is an analogous terminology for locally finite Lie algebras. A Lie subalgebra $L'$ in a locally finite Lie algebra $L$ is called \textit{locally solvable} if every finite-dimensional Lie subalgebra of $L'$ is solvable. The \textit{locally solvable radical} of a locally finite Lie algebra $L$ is defined to be the unique maximal locally solvable ideal in $L$. Such a radical is equal to the sum of all locally solvable ideals in $L$. A locally finite Lie algebra is said to be \textit{semisimple} if it has zero locally solvable radical (see \cite[p. 37]{Bo_penkov-hoyt-CLA}).
		
		In what follows, we show that the two conditions “local Lie solvability” and “Lie solvability” are equivalent in the context of locally finite Leavitt path algebras. To obtain this, we need to present some results concerning Lie solvability of $L_K(E)$. Define $\mathbf{C}_J$ to be the \textit{clock graph} pictured as follows:\\
		
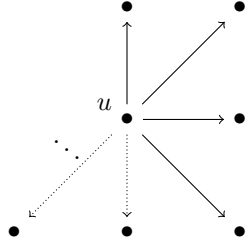
		\begin{center}
			\begin{tikzpicture}
				\node at (0,0) (O) {$\bullet$};
				\node at (-0.3,0.2) (u) {$u$};
				\node at (0,1.5) (D) {$\bullet$};
				\node at (1.5,1.5) (M) {$\bullet$};
				\node at (1.5,0) (A) {$\bullet$};
				\node at (-1.5,-1.5) (P) {$\bullet$};
				\node at (-0.8,-0.3) (dots) {$\ddots$};
				\node at (0,-1.5) (C) {$\bullet$};
				\node at (1.5,-1.5) (N) {$\bullet$};
				
				\draw [->] (O) to (A);
				\draw [->] (O) to (M);
				\draw [->] (O) to (D);
				\draw [->] (O) to (N);
				\draw [->,densely dotted] (O) to (C);
				\draw [->,densely dotted] (O) to (P);
			\end{tikzpicture}
		\end{center}
		where $r(s^{-1}(u))=\{v_j|j\in J\}$ is a non-empty set of arbitrary cardinality.
		The following lemma is a consequence of \cite[Theorem 2.3]{Pa_nam-zhang-22}.
		
		\begin{lemma}\label{lemma_4.6}
			Let $E$ be an acyclic graph and $K$ a field. The following assertions hold:
			\begin{enumerate}[font=\normalfont]
				\item If ${\rm char}(K) = 2$, then $L_K(E)$ is Lie solvable if and only if $E$ is a disjoint union of isolated vertices and graphs of type $\mathbf{C}_J$.
				\item If ${\rm char}(K) \ne 2$, then $L_K(E)$ is Lie solvable if and only if $E$ is a disjoint union of isolated vertices.
			\end{enumerate}
		\end{lemma}
		
		\begin{proposition}\label{proposition_4.7}
			Let $E$ be an acyclic graph and $K$ a field. If $L_K(E)$ contains a non-zero Lie solvable ideal, then $E$ contains a sink.
		\end{proposition}
		\begin{proof}
			Let $I$ be a non-zero Lie solvable ideal of $L_K(E)$. Put $H=E^0\cap I$ and $S=\{v\in B_H : v^H\in I\}$. Because $E$ contains no cycle, the ideal $I$ is necessarily graded. It follows from \cite[Theorem 3.3]{Pa_vas-22} that $I$ is graded isomorphic as $K$-algebra to $L_K(P_{(H,S)})$, where $P_{(H,S)}$ is the porcupine graph of the admissible pair $(H,S)$. Because $I$ is Lie solvable, it follows that $L_K(P_{(H,S)})$ is Lie solvable. Therefore, $P_{(H,S)}$ must satisfy all requirements of  Lemma \ref{lemma_4.6}. In particular, we conclude that $P_{(H,S)}$ contains a sink. Because of the constructions of $P_{(H,S)}$, every sink in $P_{(H,S)}$ is also a sink in $E$. It follows that $E$ also contains a sink.
		\end{proof}
		
		\begin{theorem}\label{theorem_4.8}
			Let $E$ be an acyclic graph and $K$ a field. If $L_K(E)$ is locally Lie solvable, then it is Lie solvable.
		\end{theorem}
		\begin{proof}
			Assume that $L_K(E)$ is locally Lie solvable. Let $F$ be an arbitrary finite subgraph of $E$. In view of Lemma \ref{lemma_4.3}, it follows that $F$ is contained in the finite graph $\dot{F}$ such that there exists an injective $K$-algebra homomorphism $L_K(\dot{F}) \to L_K(E)$. As $\dot{F}$ is finite and acyclic, by \cite[Proposition 1.5.21]{Bo_abrams-ara-molina-LPA}, we get that $L_K(\dot{F})$ is finite-dimensional over $K$. Because $L_K(\dot{F})$ is embedded in $L_K(E)$, we obtain that $L_K(\dot{F})$ is Lie solvable. Thus, by Lemma \ref{lemma_4.6}, the graph $\dot{F}$ , and hence $F$, is a disjoint union of a finite number of isolated vertices and, if ${\rm char}(K)=2$, graphs of type $\mathbf{C}_J$ with finite $|J|$. It follows that $E$ is a disjoint union of isolated vertices and, if ${\rm char}(K)=2$, graphs of type $\mathbf{C}_J$. Again, by Lemma \ref{lemma_4.6}, we conclude that $L_K(E)$ is Lie solvable. The theorem is proved.
		\end{proof}
		As a consequence, we have the following corollary whose proof follows immediately from \cite[Theorem 3.3]{Pa_vas-22} and Theorem \ref{theorem_4.8}.
		\begin{corollary}\label{corollary 4.9}
			Let $E$ be an acyclic graph and $K$ a field. Let $I$ be a graded ideal of $L_K(E)$. If I is locally Lie solvable, then it is Lie solvable.
		\end{corollary}
		We now have all the preparations in place to obtain the main result of this section which gives a description of the structure of the Lie algebra $L_K(E)^-$ associated to a Leavitt path algebra $L_K(E)$ in which every Lie ideal is an ideal.
		\begin{theorem}
			Let $E$ be a graph and $K$ a field. Assume that $L_K(E)$ has the property that every Lie ideal is an ideal. Then ${\rm char}(K)=p>0$ and $E$ satisfies the condition $(1)$, $(2)$ or $(3)$ of Theorem \ref{theorem_3.8}, and $L_K(E)^-$ is a locally finite, infinite-dimensional Lie algebra over $K$. Moreover, the following assertions hold.
			\begin{enumerate}[font=\normalfont]
				\item[(i)] If $E$ satisfies $(1)$ or $(3)$, then the locally solvable radical of $L_K(E)^-$ is $0$.
				\item[(ii)] If $E$ satisfies $(2)$, then the locally solvable radical of $L_K(E)^-$ is isomorphic to $K$.
			\end{enumerate}
		\end{theorem}
		\begin{proof}
			The fact that $E$ satisfies (1), (2) or (3) of Theorem \ref{theorem_3.8} follows immediately from the same theorem. Then, $E$ is an acyclic graph, and thus, by Corollary \ref{corollary_4.5}, we get that $L_K(E)^-$ is a locally finite Lie algebra over $K$. To prove (i), we suppose that $E$ satisfies (1) or (3) of Theorem \ref{theorem_3.8}. Let $I$ be a non-zero locally solvable Lie ideal of $L_K(E)$. As every Lie ideal of $L_K(E)$ is also an ideal, in view of \cite[Theorem 3.3.11]{Bo_abrams-ara-molina-LPA}, we obtain that $I$ is a graded ideal. Moreover, by Corollary \ref{corollary 4.9}, it is Lie solvable. As $E$ is acyclic, in view of Proposition \ref{proposition_4.7}, it follows that $E$ contains at least a sink, a contradiction. This reveals that $L_K(E)$ contains no non-zero locally solvable ideals, and thus (i) holds. Now, assume that $E$ satisfies (2) of Theorem \ref{theorem_3.8}. Then, by the same argument, it is easy to see that $Kw$ is the unique locally solvable Lie ideal of $L_K(E)$. It follows that the locally solvable radical of $L_K(E)^-$ is $Kw$ which is isomorphic to $K$. The assertion (ii) is proved. The proof of our theorem is now complete.
		\end{proof}
		The following corollary follows immediately, providing two classes of infinite-dimensional Lie algebras which are semisimple.
		\begin{corollary}\label{corollary_4.11}
			Let $K$ be a field with ${\rm char}(K)=p>0$. If $E$ is a graph satisfying the condition $(1)$ or $(3)$ of Theorem \ref{theorem_3.8}, then $L_K(E)^-$ is a locally finite, infinite-dimensional, semisimple Lie algebra over $K$.
		\end{corollary}

		\begin{remark}
			A Lie algebra $L$ over a field is called \textit{simple} if it has no ideals other than $0$ and $L$ and is not abelian. It is well-known from Lie theory that every finite-dimensional semisimple Lie algebra over a field of characteristic zero is a direct sum of simple Lie algebras. At the other extreme, in Example \ref{exmaple_4}, because the graph $E$ satisfies the condition (2) of Theorem \ref{theorem_3.8}, in view of Corollary \ref{corollary_4.11}, we get that the Lie algebra $L_K(E)^-$ of $L_K(E)$ obtained in Example \ref{exmaple_4} is an infinite-dimensional semisimple Lie algebra. Also in the same example, we have proved that $L_K(E)^-$ contains a non-trivial Lie ideal, which shows that the Lie algebra $L_K(E)^-$ is not simple. Therefore, we may use Corollary \ref{corollary_4.11} to construct a lot of examples of semisimple Lie algebras which are not direct sums of simple Lie algebras.  
		\end{remark}
	
	{\noindent\textbf{Acknowledgements} }This paper was carried out when the author was working as a postdoctoral researcher at the Vietnam Institute for Advanced Study in Mathematics (VIASM). He would like to express his warmest thanks to the VIASM for fruitful research environment and hospitality. The author would like to thank anonymous referees for carefully reading the manuscript and providing excellent suggestions for improvement. \\
	
	{\noindent\textbf{Funding} }This research is funded by the Vietnam Ministry of Education and Training under grant number B2024-CTT-02.
	
	\bigskip 
	
	{\noindent\textbf{Conflict of Interest.} }The authors have no conflict of interest to declare that are relevant to this article.
	
\end{document}